\documentclass[12pt,reqno]{article}

\usepackage[usenames]{color}
\usepackage{amssymb}
\usepackage{amsmath}
\usepackage{amsthm}
\usepackage{amsfonts}
\usepackage{amscd}
\usepackage{graphicx}

\usepackage[colorlinks=true,
linkcolor=webgreen,
filecolor=webbrown,
citecolor=webgreen]{hyperref}

\definecolor{webgreen}{rgb}{0,.5,0}
\definecolor{webbrown}{rgb}{.6,0,0}

\usepackage{color}
\usepackage{fullpage}
\usepackage{float}

\usepackage{graphics}
\usepackage{latexsym}

\DeclareMathOperator{\Li}{Li}
\DeclareMathOperator{\Cl}{Cl}

\begin{document}

\theoremstyle{plain}
\newtheorem{theorem}{Theorem}
\newtheorem{corollary}[theorem]{Corollary}
\newtheorem{proposition}{Proposition}
\newtheorem{lemma}{Lemma}
\newtheorem*{example}{Examples}
\newtheorem*{remark}{Remark}

\begin{center}
\vskip 1cm{\LARGE\bf 
Evaluation of harmonic number series involving the binomial coefficient $C(3n,n)$ in the denominator by integration \\
}
\vskip 1cm
{\large

Kunle Adegoke \\
Department of Physics and Engineering Physics \\ Obafemi Awolowo University, 220005 Ile-Ife \\ Nigeria \\
\href{mailto:adegoke00@gmail.com}{\tt adegoke00@gmail.com}

\vskip 0.7cm

Robert Frontczak \\
Independent Researcher, 72762 Reutlingen \\ Germany \\
\href{mailto:robert.frontczak@web.de}{\tt robert.frontczak@web.de}

}

\end{center}

\vskip .2 in

\begin{abstract}
Two classes of infinite series involving harmonic numbers and the binomial coefficient $C(3n,n)$ are evaluated in closed form using integrals. 
Several remarkable integral values and difficult series identities are stated as special cases of the main results. 
\end{abstract}



\bigskip

\section{Introduction, Motivation and Preliminaries}

Integration is a classical tool in the evaluation of infinite series. The approach was revived recently 
in the papers by Sofo and Nimbran \cite{Sofo}, Stewart \cite{Stewart}, and Li and Chu \cite{Li1,Li2,Li3,Li4}.
For instance, Li and Chu review in \cite{Li4} a few (known) series involving harmonic numbers $H_n$ and odd harmonic numbers $O_n$  
such as
\begin{equation*}
\sum_{n=1}^\infty \frac{H_n}{n^2}, \qquad \sum_{n=1}^\infty (-1)^{n-1} \frac{H_n}{n^2}, \qquad \sum_{n=1}^\infty \frac{O_n}{n^2},
\qquad \sum_{n=1}^\infty (-1)^{n-1} \frac{O_n}{n^2},
\end{equation*}
obtained by using definite integrals. The authors evaluate difficult addition series involving alternating harmonic and odd harmonic numbers in closed form by employing calculus and complex analysis. Here, as usual, harmonic numbers $H_n$ and odd harmonic numbers $O_n$ are defined by $H_0=0$, $O_0=0$, and 
\begin{equation*}
H_n =\sum_{k=1}^n \frac{1}{k},\qquad O_n =\sum_{k=1}^n \frac{1}{2k-1}.
\end{equation*} 
Obvious relations between harmonic numbers $H_n$ and odd harmonic numbers $O_n$ are the following:
\begin{equation*}
H_{2n} = \frac{1}{2} H_n + O_n \qquad \text{and} \qquad H_{2n - 1} = \frac{1}{2}H_{n - 1} + O_n.
\end{equation*}

In this paper, we proceed in the same direction. Using integrals in combination with complex analysis and partial fraction decompositions 
we will evaluate in closed form the Euler-type series
\begin{equation*}
\sum_{k=0}^\infty \frac{H_{3k+1} - H_k}{(3k+1) \binom{3k}{k}} \binom{k}{m} z^k \qquad \text{and} \qquad  
\sum_{k=0}^\infty \frac{H_{2k} - H_k}{(3k+1) \binom{3k}{k}} \binom{k}{m} z^k,
\end{equation*}
for all $m\geq 0$ and all $z\in\mathbb{C}$ with $|z|<1$. For instance, we will prove that
\begin{equation*}
\sum_{k=0}^\infty \frac{H_{3k+1} - H_{k}}{(3k+1) \binom{3k}{k}\,2^{k + 1}} = \frac{\pi^2}{48} - \frac{\ln^2 2}{10} + \frac{2}{5}\,G,
\end{equation*}
where $G=\sum_{j=0}^\infty (-1)^j / (1+2j)^2$ is Catalan's constant. Another difficult evaluation that will be derived is
\begin{equation*}
\begin{split}
\sum_{k=0}^\infty \frac{H_{2k} - H_k}{(3k + 1) \binom{3k}{k}} \frac{(- 1)^k k}{4^{k + 1}} 
&= \frac{3}{{448}}\ln 2 - \frac{9}{{512}}\ln^2 2 - \frac{3}{{128}}\arctan^2 \left( {\frac{{\sqrt 7 }}{5}} \right)\\
&\qquad - \left( {\frac{1}{{224}} - \frac{{89}}{{6272}}\ln 2} \right)\sqrt 7 \arctan \left( {\frac{{\sqrt 7 }}{5}} \right).
\end{split}
\end{equation*}

We proceed wit two special functions that will be needed. Let $\Li_2(z)$ be the dilogarithm defined by (see Lewin \cite{Lewin})
\begin{equation*}
\Li_2(z) = \sum_{k=1}^\infty \frac{z^k}{k^2}, \qquad |z|\leq 1,
\end{equation*}
having the special values
\begin{equation}\label{eq.k4mxnco}
\Li_2(1) = \frac{\pi^2}{6}, \quad\text{and}\quad \Li_2\left (\frac{1}{2}\right) = \frac{\pi^2}{12} - \frac{\ln^2(2)}{2}.
\end{equation}
Let also $\Cl_2(z)$ be the Clausen's function defined by~\cite{Lewin,tric}
\begin{equation*}
\Cl_2 (z) = \sum_{n = 1}^\infty \frac{\sin (nz)}{n^2} = - \int_0^z \ln |2\sin (\theta /2)|d\theta.
\end{equation*}
This function has the functional relations
\begin{gather}
\Cl_2 (\pi  + \theta ) = - \Cl_2 (\pi - \theta ),\\
\Cl_2 (\theta ) = - \Cl_2 (2\pi - \theta ),\\
\frac{1}{2}\Cl_2 (2\theta ) = \Cl_2 (\theta ) - \Cl_2 (\pi - \theta )\label{eq.rqfmoa7};
\end{gather}
and the special values
\begin{equation}
\Cl_2(n\pi)=0,\quad n\in\mathbb Z^+,
\end{equation}
and
\begin{equation}\label{eq.v1bi0dh}
\Cl_2(\pi/2)=G=-\Cl_2(3\pi/2),
\end{equation}
where $G$ is Catalan's constant. 

We conclude this section with a motivation of our approach. We start with the Beta integral \cite{Srivastava}:
\begin{equation*}
\int_0^1 x^{a-1} (1-x)^{b-1} dx = B(a,b) = \frac{\Gamma(a) \Gamma(b)}{\Gamma(a+b)}, \qquad a,b>0.
\end{equation*}
Differentiating the above definition with respect to $a$ and using the fact that
\begin{equation*}
\frac{d}{da} x^{a-1} = x^{a-1}\ln(x)
\end{equation*}
we get
\begin{equation*}
\int_0^1 x^{a-1} (1-x)^{b-1} \ln(x) dx = \frac{\Gamma(a) \Gamma(b)}{\Gamma(a+b)}\big (\psi(a) - \psi(a+b) \big),
\end{equation*}
where $\psi(x)=\Gamma'(x)/\Gamma(x)$ is the psi or digamma function, $\Gamma(x)$ being the Gamma function.
This function is related to harmonic numbers via $\psi(n+1)=H_n-\gamma$, where $\gamma$ is the Euler-Mascheroni constant. 
From here we can make the transformations $a\mapsto ka+1$ and $b\mapsto 2kb+1$ to obtain
\begin{align*}
\int_0^1 x^{ka} (1-x)^{2kb} \ln(x) dx &= \frac{\Gamma(ka+1) \Gamma(2kb+1)}{\Gamma(k(a+2b)+2)}\Big (\psi(ka+1) - \psi(k(a+2b)+2) \Big) \\
&= \frac{H_{ka} - H_{k(a+2b)+1}}{(k(a+2b)+1) \binom{k(a+2b)}{ka}}.
\end{align*}
Also, by symmetry or by applying the transformations $a\mapsto 2ka+1$ and $b\mapsto kb+1$ we obtain
\begin{equation*}
\int_0^1 x^{2ka} (1-x)^{kb} \ln(x) dx = \frac{H_{2ka} - H_{k(2a+b)+1}}{(k(2a+b)+1) \binom{k(2a+b)}{kb}}.
\end{equation*}
Now, let $a=b=1$. Then
\begin{equation*}
\int_0^1 x^{k} (1-x)^{2k} \ln(x) dx = \frac{H_{k} - H_{3k+1}}{(3k+1) \binom{3k}{k}},
\end{equation*}
and we can consider the series (for all $z\in\mathbb{C}$ with $|z|\leq 1$)
\begin{equation}\label{eq.tkvzjtd}
\sum_{k=0}^\infty \frac{H_{3k+1} - H_{k}}{(3k+1) \binom{3k}{k}} z^k = - \int_0^1 \frac{\ln(x)}{1-zx(1-x)^2} dx, 
\end{equation}
which can also be written as
\begin{equation}\label{eq.gzvmjym}
\sum_{k=0}^\infty \frac{H_{3k+1} - H_{k}}{(3k+1) \binom{3k}{k}}\,\frac{1}{z^{k+1}} = \int_0^1 \frac{{\ln x}}{{x(1-x)^2 - z}}\,dx,\quad |z|\ge 1.
\end{equation}
In particular,
\begin{equation}\label{eq.p1bdo31}
\sum_{k=0}^\infty \frac{H_{3k+1} - H_{k}}{(3k+1) \binom{3k}{k}}\,\frac{1}{2^{k + 1}} = \int_0^1 \frac{{\ln (x)}}{(x - 2)(x^2 + 1)}\,dx
\end{equation}
and
\begin{equation}\label{eq.r39hp67}
\sum_{k=0}^\infty \frac{H_{3k+1} - H_{k}}{(3k+1) \binom{3k}{k}}\,\frac {(-1)^k}{4^{k + 1}} = -\int_0^1 \frac{{\ln (x)}}{(x + 1)(x^2-3x+4)}\,dx. 
\end{equation}
Similarly, we obtain 
\begin{align*}
\int_0^1 x^{2ka} (1-x)^{ka} \ln(x) dx &= - \int_1^0 (1-x)^{2ka}\,x^{ka} \ln(1-x) dx \\
&= \frac{H_{2ka} - H_{k(2a+b)+1}}{(k(2a+b)+1) \binom{k(2a+b)}{kb}}.
\end{align*}
This gives for all $z\in\mathbb{C}$ with $|z|\leq 1$
\begin{equation*}
\sum_{k=0}^\infty \frac{H_{ka} - H_{2ka}}{(3ka+1) \binom{3ka}{ka}} z^k = \int_0^1 \frac{\ln\left (\frac{x}{1-x}\right )}{1-zx^a(1-x)^{2a}} dx
\end{equation*}

or the particular relation

\begin{equation}\label{eq.wl9np1c}
\sum_{k=0}^\infty \frac{H_{2k} - H_{k}}{(3k+1) \binom{3k}{k}} z^k = - \int_0^1 \frac{\ln\left (\frac{x}{1-x}\right )}{1-zx(1-x)^{2}} dx.
\end{equation}

The evaluation of the integrals is not trivial but can be done by applying some additional theory.

\section{Main results, Part 1}

In this section, we explicitly deal with the series on the left hand side of \eqref{eq.tkvzjtd}.

\begin{lemma}
For $\lambda\not\in [0,1)$ we have
\begin{gather}
\int_0^1 \frac{{\ln x}}{{x - \lambda }}\,dx = \Li_2 \left( {\frac{1}{\lambda }} \right),\label{eq.taczr9z}\\
\int_0^1 \frac{{\ln (x/(1 - x))}}{{x - \lambda }}\,dx = - \frac{1}{2}\ln^2 \left( {\frac{{\lambda - 1}}{\lambda }} \right) \label{eq.yqwhh2z}.
\end{gather}
\end{lemma}
\begin{proof}
Identity~\eqref{eq.taczr9z} follows immediately from the fact that
\begin{equation*}
\int {\frac{{\ln x}}{{x - \lambda }}\,dx} = {\rm Li}_2 \left( {\frac{x}{\lambda }} \right) + \ln x\ln \left( {\frac{{\lambda - x}}{\lambda }} \right) + const.
\end{equation*}
Now,
\begin{equation}
\int_0^1 \frac{{\ln (x/(1 - x))}}{{x - \lambda }}\,dx = \int_0^1 \frac{{\ln x}}{{x - \lambda }}\,dx - \int_0^1 \frac{{\ln (1 - x)}}{{x - \lambda }}\,dx.
\end{equation}
Let
\begin{equation*}
I = \int_0^1 \frac{{\ln (1 - x)}}{{x - \lambda }}\,dx.
\end{equation*}
A change of variable $u=1 - x$ gives
\begin{equation*}
I = - \int_0^1 \frac{{\ln u}}{{u - (1 - \lambda )}}\,du = - \Li_2 \left( {\frac1{{1 - \lambda }}} \right)
\end{equation*}
on account of~\eqref{eq.taczr9z}. Thus
\begin{equation*}
\int_0^1 {\frac{{\ln (x/(1 - x))}}{{x - \lambda }}\,dx} = \Li_2 \left( {\frac{1}{\lambda }} \right) + \Li_2 \left( {\frac1{{1 - \lambda }}} \right) = - \frac12\ln^2 \left( {\frac{{\lambda - 1}}{\lambda }} \right)
\end{equation*}
since~\cite[(8), p.~283]{Lewin}
\begin{equation*}
\Li_2 (x) + \Li_2 \left( {\frac{x}{{x - 1}}} \right) = - \frac{1}{2}\ln^2 (1 - x),\quad x < 1.
\end{equation*}
\end{proof}

\begin{lemma}[{\cite[p.~291]{Lewin}}]
For $\theta\in [-2\pi,2\pi]$, we have
\begin{equation}\label{eq.lk8zgeb}
\Li_2 \left( {e^{i\theta } } \right) = \frac{{\pi ^2 }}{6} + \frac{{\theta^2 - 2\pi |\theta |}}{4} + i\Cl_2 (\theta ),
\end{equation}
where $\Cl_2(z)$ is Clausen's function.
\end{lemma}

\begin{proposition}\label{prop.nxjp5ek}
We have 
\begin{equation}\label{eq.grfgv7c}
\sum_{k=0}^\infty \frac{H_{3k+1} - H_{k}}{(3k+1) \binom{3k}{k}\,2^{k + 1}} = \frac{\pi^2}{48} - \frac{\ln^2(2)}{10} + \frac{2}{5}\, G.
\end{equation}
\end{proposition}
\begin{proof}
Using the decomposition 
\begin{equation*}\label{eq.f2814qz}
\frac1{{(x - 2)(x^2  + 1)}} = \frac1{{(4i - 2)}}\,\frac1{{x + i}} - \frac1{{(4i + 2)}}\,\frac1{{x - i}} + \frac15\,\frac{1}{{x - 2}}
\end{equation*}
we have
\begin{equation*}\label{eq.czhzvt3}
\begin{split}
\int_0^1 \frac{{\ln x}}{{(x - 2)(x^2  + 1)}}\,dx &= \frac{1}{{4i - 2}}\int_0^1 \frac{{\ln x}}{{x + i}}\,dx 
- \frac{1}{{4i + 2}}\int_0^1 \frac{{\ln x}}{{x - i}}\,dx + \frac{1}{5}\int_0^1 \frac{{\ln x}}{{x - 2}}\,dx \\
& = \frac{1}{{4i - 2}}\Li_2 \left( {e^{i\pi /2} } \right) - \frac{1}{{4i + 2}}\Li_2 \left( {e^{ - i\pi /2} } \right) 
+ \frac{1}{5}\Li_2 \left( {\frac{1}{2}} \right),
\end{split}
\end{equation*}
in view of~\eqref{eq.taczr9z}. Now,~\eqref{eq.lk8zgeb} gives
\begin{equation}\label{eq.p9szd4d}
\Li_2 (e^{i\pi /2} ) = - \frac{{\pi ^2 }}{{48}} + iG,\quad \Li_2 (e^{- i\pi /2} ) = - \frac{{\pi ^2 }}{{48}} + i\Cl_2 (- \pi /2).
\end{equation}
Using~\eqref{eq.rqfmoa7} with $\theta=-\pi/2$ gives
\begin{equation}
\Cl_2(-\pi/2) = \Cl_2(3\pi/2) = -G,\qquad \text{by~\eqref{eq.v1bi0dh}};
\end{equation}
so that
\begin{equation}\label{eq.mfthsdb}
\Li_2 (e^{ - i\pi /2} ) = - \frac{{\pi ^2 }}{{48}} - iG.
\end{equation}
Thus, using~\eqref{eq.p9szd4d},~\eqref{eq.mfthsdb} and the evaluation of $\Li_2(1/2)$ from~\eqref{eq.k4mxnco}, we have
\begin{equation*}
\int_0^1 \frac{{\ln x}}{{(x - 2)(x^2  + 1)}}\,dx = \frac{1}{{4i - 2}}\left( { - \frac{{\pi ^2 }}{{48}} + iG} \right) 
- \frac{1}{{4i + 2}}\left(- \frac{{\pi^2 }}{{48}} - iG \right) + \frac{1}{5}\left( \frac{{\pi^2 }}{{12}} - \frac{1}{2}\ln^2 2 \right)
\end{equation*}
and hence~\eqref{eq.grfgv7c}, in view of~\eqref{eq.p1bdo31}.
\end{proof}

\begin{proof}[Alternative proof of Proposition~\ref{prop.nxjp5ek}] 
We begin with the partial fraction decomposition
\begin{equation*}
\frac{1}{(x-2)(x^2+1)} = \frac{1}{5} \Big ( \frac{1}{x-2} - \frac{x+2}{x^2+1}\Big ).
\end{equation*}
Hence, 
\begin{equation*}
\int_0^1 \frac{\ln x}{(x-2)(x^2+1)}dx = \frac{1}{5} \Big ( \int_0^1 \frac{\ln x}{x-2} dx - \int_0^1 \frac{(x+2)\ln x}{x^2+1} dx \Big ).
\end{equation*}
As
\begin{equation*}
\int \frac{\ln x}{x-2} dx = \ln(2)\ln(|x-2|) - \Li_2\left (\frac{2-x}{2}\right) + const,
\end{equation*}
we get
\begin{equation*}
\int_0^1 \frac{\ln x}{x-2} dx = \frac{\pi^2}{12} - \frac{\ln^2(2)}{2}.
\end{equation*}
Next,
\begin{equation*}
\int \frac{x\ln x}{x^2+1} dx = \frac{1}{2} \Big ( \Li_2(ix) + \Li_2(-ix) + \ln(x)\ln(1+x^2)\Big ) + const.,
\end{equation*}
and this gives
\begin{equation*}
\int \frac{x\ln x}{x^2+1} dx = \frac{1}{2} \Big ( \Li_2(i) + \Li_2(-i) \Big ) = - \frac{\pi^2}{48},
\end{equation*}
where the relation (see also \cite{Lewin})
\begin{equation*}
\Li_2(x) + \Li_2(-x) = \frac{1}{2} \Li_2(x^2),
\end{equation*}
was used. Finally, from the indefinite integral
\begin{equation*}
2 \int \frac{\ln x}{x^2+1} dx = i \Big ( \Li_2(ix) - \Li_2(-ix) + \ln(x)\ln\Big(\frac{i+x}{i-x}\Big )\Big ) + const,
\end{equation*}
we get
\begin{equation*}
2 \int_0^1 \frac{\ln x}{x^2+1} dx = i ( \Li_2(i) - \Li_2(-i)) = -2 \sum_{j=0}^\infty \frac{(-1)^j}{(1+2j)^2} = -2 G. 
\end{equation*}
Putting everything together we obtain the claimed result.
\end{proof}

\begin{proposition}\label{prop.vhpcqxx}
We have
\begin{equation}\label{eq.fx7c65q}
\sum_{k = 0}^\infty  {\frac{{H_{3k + 1}  - H_k }}{{(3k + 1)\binom{{3k}}{k}}}\frac{{( - 1)^k }}{{4^{k + 1} }} = \frac{{\pi ^2 }}{{96}} + \frac{1}{8}\,\Re\Li_2 \left( {\frac{3}{8} + \frac{{i\sqrt 7 }}{8}} \right)}  + \frac{{5\sqrt 7 }}{{56}}\,\Im\Li_2 \left( {\frac{3}{8} + \frac{{i\sqrt 7 }}{8}} \right).
\end{equation}
\end{proposition}
\begin{proof}
Consider the partial fraction decomposition
\begin{equation*}
\begin{split}
\frac{1}{{x(1 - x)^2 + 4}} &= \frac{1}{{(x + 1)(x^2 - 3x + 4)}}\\
&= \frac18\,\frac{1}{{x + 1}} - \frac{1}{{112}}\,\frac{{7 + 5i\sqrt 7 }}{{x - (3 + i\sqrt 7 )/2}} 
- \frac{1}{{112}}\,\frac{{7 - 5i\sqrt 7}}{{x - (3 - i\sqrt 7 )/2}};
\end{split}
\end{equation*}
which allows the integral on the rhs of~\eqref{eq.r39hp67} to be written as
\begin{equation*}
\begin{split}
\int_0^1 \frac{{\ln x\,dx}}{{(x + 1)(x^2 - 3x + 4)}} &= \frac{1}{8}\int_0^1 \frac{{\ln x\,dx}}{{x + 1}} 
- \frac{{7 + 5i\sqrt 7 }}{{112}}\int_0^1 \frac{{\ln x\,dx}}{{x - (3 + i\sqrt 7 )/2}} \\ 
&\qquad - \frac{{7 - 5i\sqrt 7 }}{{112}}\int_0^1 \frac{{\ln x\,dx}}{{x - (3 - i\sqrt 7 )/2}},
\end{split}
\end{equation*}
which upon using~\eqref{eq.taczr9z} gives
\begin{equation}\label{eq.xa20o4n}
\begin{split}
\int_0^1 \frac{{\ln x\,dx}}{{(x + 1)(x^2 - 3x + 4)}} &= \frac{1}{8}\,\Li_2 ( - 1) 
- \frac{{7 + 5i\sqrt 7 }}{{112}}\,\Li_2 \left( {\frac{2}{{3 + i\sqrt 7 }}} \right)\\
&\qquad - \frac{{7 - 5i\sqrt 7 }}{{112}}\,\Li_2 \left( {\frac{2}{{3 - i\sqrt 7 }}} \right)\\
&= \frac{1}{8}\Li_2 ( - 1) - \frac{1}{{56}}\Re\left( {(7 + 5i\sqrt 7 )\Li_2 \left( {\frac{2}{{3 + i\sqrt 7 }}} \right)} \right).
\end{split}
\end{equation}
Use of~\eqref{eq.xa20o4n} in~\eqref{eq.r39hp67} gives~\eqref{eq.fx7c65q}, after simplification. In taking the real part in~\eqref{eq.xa20o4n}, it is convenient to use the fact that $\Re (fg)=\Re f\Re g - \Im f\Im g$ for arbitrary $f$ and $g$.
\end{proof}

Differentiating~\eqref{eq.tkvzjtd} $m$ times with respect to $z$ and thereafter replacing $z$ with $1/z$ gives
\begin{equation}\label{eq.u0vzry8}
\sum_{k=0}^\infty \frac{H_{3k+1} - H_{k}}{(3k+1) \binom{3k}{k}} \frac {\binom km}{z^{k + 1}} =(-1)^m\int_0^1 {\frac{{x^m(1 - x)^{2m}\ln x}}{\left(x(1 - x)^2  - z\right)^{m + 1}}\,dx},\quad |z|\ge 1. 
\end{equation}

\begin{theorem}\label{thm.oxop1ki}
Let $m$ be a non-negative integer and let $z_1$, $z_2$ and $z_3$ be the distinct roots of $x(1 - x)^2-z=0$ where $z$ is a real number such that $|z|\ge1$. Then
\begin{equation}\label{eq.uhuus5d}
\sum_{k=0}^\infty \frac{H_{3k+1} - H_{k}}{(3k+1) \binom{3k}{k}} \frac {\binom{k}{m}}{z^{k + 1}} 
= (-1)^m \sum_{k=1}^3 {\sum_{j=0}^m {a_j{(z_k)}C_j(z_k)}}; 
\end{equation}
where, for $0\le r\le m$,
\begin{equation}\label{eq.rkhw1l3}
a_r{(\lambda)} = \frac{1}{{(m - r)!}}\left. {\frac{{d^{m - r} }}{{dx^{m - r} }}\frac{x^m(1 - x)^{2m}(x - \lambda)^{m + 1}}{{\left(x(1 - x)^2 - z\right)^{m + 1} }}} \right|_{x = \lambda} ,
\end{equation}
\begin{equation}\label{eq.w9cngvk}
C_0 (\lambda )=\Li_2 \left(\frac1\lambda \right),
\end{equation}
and for $r$ a positive integer and $\lambda\not\in [0,1]$,
\begin{equation}\label{eq.y3tw4qr}
C_r (\lambda )= \frac{{( - 1)^r }}{r}\sum_{p = 1}^{r - 1} {\frac{1}{{p\lambda ^{r - p} }}\left( {\frac{1}{{(\lambda  - 1)^p }} - \frac{1}{{\lambda ^p }}} \right)}  - \frac{{( - 1)^r }}{{r\lambda ^r }}\ln \left( {\frac{{\lambda  - 1}}{\lambda }} \right).
\end{equation}
\end{theorem}
\begin{proof}
We start with the partial fraction decomposition
\begin{equation}
\begin{split}
\frac{x^m(1 - x)^{2m}}{{\left( {x(1 - x)^2 - z} \right)^{m + 1} }} &= \sum_{j = 0}^m {\left( {\frac{{a_j{(z_1)} }}{{(x - z_1)^{j + 1} }} + \frac{{a_j{(z_2)} }}{{(x - z_2)^{j + 1} }} + \frac{{a_j{(z_3)} }}{{(x - z_3)^{j + 1} }}} \right)}\\ 
& = \sum_{k = 1}^3 {\sum_{j = 0}^m {\frac{{a_j{(z_k)} }}{{(x - z_k )^{j + 1} }}} } ,
\end{split}
\end{equation}
where, for $0\le r\le m$, the coefficients $a_r{(z_1)}$, $a_r{(z_2)}$, $a_r{(z_3)}$ are found from~\eqref{eq.rkhw1l3}.

We therefore have
\begin{equation}\label{eq.n0a1snj}
\int_0^1 {\frac{{\left(x^m(1 - x)^{2m}\right)\ln x}}{{\left( {x(1 - x)^2 - z} \right)^{m + 1} }}\,dx} 
= \sum_{k = 1}^3 {\sum_{j = 0}^m {a_j{(z_k)} \int_0^1 \frac{{\ln x}}{{(x - z_k )^{j + 1} }}\,dx} }.
\end{equation}
Differentiating~\eqref{eq.taczr9z} $j$ times with respect to $\lambda$ gives
\begin{equation}\label{eq.xhcawq8}
\int_0^1 {\frac{{\ln x}}{{(x - \lambda )^{j + 1} }}\,dx} = \frac{{( - 1)^j }}{j}\sum_{p = 1}^{j - 1} {\frac{1}{{p\lambda ^{j - p} }}\left( {\frac{1}{{(\lambda  - 1)^p }} - \frac{1}{{\lambda ^p }}} \right)}  - \frac{{( - 1)^j }}{{j\lambda ^j }}\ln \left( {\frac{{\lambda  - 1}}{\lambda }} \right)=C_j (\lambda ),\quad j\ne 0,
\end{equation}
which, when utilized in~\eqref{eq.n0a1snj} yields an evaluation of the integral on the right hand side of~\eqref{eq.u0vzry8} 
and hence~\eqref{eq.uhuus5d}.
\end{proof}

\begin{corollary}\label{cor.lofauml}
Let $m$ be a non-negative integer and let $(\lambda_1,\lambda_2,\lambda_3)=(2,-i,i)$ where $i$ is the imaginary unit. Then
\begin{equation}\label{eq.eoko23c}
\sum_{k=0}^\infty \frac{H_{3k+1} - H_{k}}{(3k+1) \binom{3k}{k}} \frac {\binom km}{2^{k + 1}} =(-1)^m\sum_{k=1}^3{\sum_{j=0}^m{a_j{(\lambda_k)}C_j(\lambda_k)}}; 
\end{equation}
where, for $0\le r\le m$, the coefficients are given by
\begin{gather}
a_r{(2)}  = \frac{1}{{(m - r)!}}\left. {\frac{{d^{m - r} }}{{dx^{m - r} }}\,\frac{x^m(1 - x)^{2m}}{{(x^2  + 1)^{m + 1} }}} \right|_{x = 2} ,\label{eq.zufjtr4}\\
a_r{(-i)}  = \frac{1}{{(m - r)!}}\left. {\frac{{d^{m - r} }}{{dx^{m - r} }}\,\frac{x^m(1 - x)^{2m}}{{\left( {(x - 2)(x - i)} \right)^{m + 1} }}} \right|_{x =  - i} ,\\
a_r{(i)}  = \frac{1}{{(m - r)!}}\left. {\frac{{d^{m - r} }}{{dx^{m - r} }}\,\frac{x^m(1 - x)^{2m}}{{\left( {(x - 2)(x + i)} \right)^{m + 1} }}} \right|_{x = i}\label{eq.c1scd26};
\end{gather}
\begin{equation}\label{eq.qdvcic2}
C_0 (2)=\Li_2 \left(\frac12 \right)=\frac{\pi^2}{12} -\frac12\ln^22,\quad C_0(-i)=-\frac{\pi^2}{48} + iG,\quad C_0(i)=-\frac{\pi^2}{48} - iG,\text{ by~\eqref{eq.w9cngvk}};
\end{equation}
and for $1\le r\le m$, $C_r(2)$, $C_r(-i)$ and $C_r(i)$ are found from~\eqref{eq.y3tw4qr}.
\end{corollary}
\begin{proof}
Set $z=2$ in Theorem~\ref{thm.oxop1ki}.
\end{proof}

We now list some examples from~\eqref{eq.eoko23c}:
\begin{equation}
\sum_{k = 0}^\infty \frac{{H_{3k + 1} - H_k }}{(3k + 1)\binom {3k}{k}}\frac{k}{{2^{k + 1} }} 
= \frac{{3\pi }}{{100}} - \frac{{3\pi ^2 }}{{400}} + \frac{3}{{25}}\ln 2 + \frac{9}{{250}}\ln^2 2 - \frac{{13}}{{125}}G,
\end{equation}
\begin{equation}
\begin{split}
&\sum_{k = 0}^\infty \frac{{H_{3k + 1} - H_k }}{(3k + 1)\binom {3k}{k}} \frac{{k(k - 1)}}{{2^{k + 2} }} \\ 
&\qquad= \frac{3}{{100}} - \frac{{29\pi }}{{1250}} + \frac{{149\pi^2 }}{{30000}} - \frac{{58}}{{625}}\ln 2 - \frac{{149}}{{6250}}\ln^2 2 + \frac{{243}}{{3125}}G,
\end{split}
\end{equation}
\begin{equation}
\begin{split}
&\sum_{k = 0}^\infty \frac{{H_{3k + 1} - H_k }}{(3k + 1)\binom {3k}{k}} \frac{{\binom{{k}}{3}}}{{2^{k + 1} }} \\ 
&\qquad = - \frac{{13}}{{375}} + \frac{{1529\pi }}{{75000}} - \frac{{577\pi^2 }}{{150000}} + \frac{{752}}{{9375}}\ln 2 
+ \frac{{577}}{{31250}}\ln^2 2 - \frac{{1903}}{{31250}}G.
\end{split}
\end{equation}

\begin{corollary}\label{cor.p7i3dyq}
Let $m$ be a non-negative integer and let \mbox{$(\gamma_1,\gamma_2,\gamma_3)=(-1,(3 + i\sqrt 7)/2,(3 - i\sqrt 7)/2)$} where $i$ is the imaginary unit. Then
\begin{equation}\label{eq.bcdy054}
\sum_{k=0}^\infty \frac{H_{3k+1} - H_{k}}{(3k+1) \binom{3k}{k}}\, \frac {(-1)^k\binom km}{4^{k + 1}} =(-1)^{m + 1}\sum_{k=1}^3{\sum_{j=0}^m{a_j{(\gamma_k)}C_j(\gamma_k)}}; 
\end{equation}
where, for $0\le r\le m$, the coefficients are given by
\begin{gather}
a_r{(\gamma_1)}  = \frac{1}{{(m - r)!}}\left. {\frac{{d^{m - r} }}{{dx^{m - r} }}\,\frac{x^m(1 - x)^{2m}}{{(x^2  - 3x + 4)^{m + 1} }}} \right|_{x = -1} ,\label{eq.w33l0zn}\\
a_r{(\gamma_2)}  = \frac{1}{{(m - r)!}}\left. {\frac{{d^{m - r} }}{{dx^{m - r} }}\,\frac{x^m(1 - x)^{2m}}{{\left( {(x + 1)(x - (3 - i\sqrt 7)/2)} \right)^{m + 1} }}} \right|_{x =  (3 + i\sqrt 7)/2} ,\\
a_r{(\gamma_3)}  = \frac{1}{{(m - r)!}}\left. {\frac{{d^{m - r} }}{{dx^{m - r} }}\,\frac{x^m(1 - x)^{2m}}{{\left( {(x + 1)(x - (3 + i\sqrt 7)/2)} \right)^{m + 1} }}} \right|_{x =  (3 - i\sqrt 7)/2}\label{eq.i2xtucg};
\end{gather}
\begin{equation}\label{eq.uct5elk}
C_0 (\gamma _1 ) = \Li_2 ( - 1) =  - \frac{{\pi ^2 }}{{12}},\quad C_0 (\gamma _2 ) = \Li_2 \left( {\frac{2}{{3 + i\sqrt 7 }}} \right),\quad C_0 (\gamma _3 ) = \Li_2 \left( {\frac{2}{{3 - i\sqrt 7 }}} \right),\text{ by~\eqref{eq.w9cngvk}};
\end{equation}
and for $1\le r\le m$, $C_r(\gamma_1)$, $C_r(\gamma_2)$ and $C_r(\gamma_3)$ are found from~\eqref{eq.y3tw4qr}.
\end{corollary}
\begin{proof}
Set $z=-4$ in Theorem~\ref{thm.oxop1ki}. Note that $\gamma_1=-1$, $\gamma_2=(3 + i\sqrt 7)/2$ and $\gamma_3=(3 - i\sqrt 7)/2$ 
are the roots of $x(1 - x)^2 + 4=0$.
\end{proof}

We give one example from~\eqref{eq.bcdy054}. At $m=0$ we have
\begin{equation}\label{eq.z7wnv9j}
\sum_{k = 0}^\infty \frac{{H_{3k + 1} - H_k }}{{\left( {3k + 1} \right)\binom{{3k}}{k}}}\,\frac{{( - 1)^k }}{{4^{k + 1} }} 
= \frac{{\pi ^2 }}{{96}} - \frac{4}{{\sqrt 7 }}\,\Im {\frac{{\Li_2 \left( {\frac{2}{{3 + i\sqrt 7 }}} \right)}}{{5 + i\sqrt 7 }}},
\end{equation}
since~\eqref{eq.bcdy054} at $m=0$ gives
\begin{equation*}
\sum_{k = 0}^\infty \frac{{H_{3k + 1}  - H_k }}{{\left( {3k + 1} \right)\binom{{3k}}{k}}}\frac{{( - 1)^k }}{{4^{k + 1} }} = \frac{{\pi^2 }}{{96}} + \frac{{2i}}{{\sqrt 7 }}\frac{{\Li_2 \left( {2/\left( {3 + i\sqrt 7 } \right)} \right)}}{{5 + i\sqrt 7 }} - \frac{{2i}}{{\sqrt 7 }}\frac{{\Li_2 \left( {2/\left( {3 - i\sqrt 7 } \right)} \right)}}{{5 - i\sqrt 7 }};
\end{equation*}
the right hand side of which can be simplified using
\begin{equation*}
fg + f^* g^* = 2\,\Re f\,\Re g - 2\,\Im f\, \Im g,
\end{equation*}
for arbitrary functions $f$ and $g$. \\

Differentiating~\eqref{eq.gzvmjym} $m$ times with respect to $z$ gives
\begin{equation}\label{eq.kjlniph}
\sum_{k=0}^\infty \frac{H_{3k+1} - H_{k}}{(3k+1) \binom{3k}{k}} \frac {\binom{k + m}k}{z^{k + m + 1}} =(-1)^m\int_0^1 {\frac{{\ln x}}{\left(x(1 - x)^2  - z\right)^{m + 1}}dx},\quad |z|\ge 1. 
\end{equation}
Note that~\eqref{eq.kjlniph} holds for every real number $m$ that is not a negative integer.

Setting $z=2$ in~\eqref{eq.kjlniph} gives
\begin{equation}\label{eq.aumtswz}
\sum_{k=0}^\infty \frac{H_{3k+1} - H_{k}}{(3k+1) \binom{3k}{k}} \frac {\binom{k + m}k}{2^{k + m + 1}} 
= (-1)^m\int_0^1 \frac{{\ln x}}{((x - 2)(x^2 + 1))^{m + 1}}dx. 
\end{equation}

\begin{theorem}\label{thm.mtaaocu}
Let $m$ be a non-negative integer and let $z_1$, $z_2$ and $z_3$ be the distinct roots of $x(1 - x)^2-z=0$ where $z$ is a real number 
with $|z|\ge 1$. Then
\begin{equation}\label{eq.fbxhg42}
\sum_{k=0}^\infty \frac{H_{3k+1} - H_{k}}{(3k+1) \binom{3k}{k}} \frac {\binom{k + m}k}{z^{k + m + 1}} 
= (-1)^m \sum_{k=1}^3 {\sum_{j=0}^m {b_j{(z_k)}C_j(z_k)}}; 
\end{equation}
where, for $0\le r\le m$,
\begin{equation}\label{eq.tp25jkt}
b_r{(\lambda)} = \frac{1}{{(m - r)!}}\left. {\frac{{d^{m - r} }}{{dx^{m - r} }}\frac{(x - \lambda)^{m + 1}}{{\left(x(1 - x)^2 - z\right)^{m + 1} }}} \right|_{x = \lambda} ,
\end{equation}
and $C_r(z_1)$, $C_r(z_2)$ and $C_r(z_3)$ are calculated using~\eqref{eq.w9cngvk} and~\eqref{eq.y3tw4qr}.
\end{theorem}
\begin{proof}
Consider the partial fraction decomposition
\begin{equation}
\begin{split}
\frac{1}{{\left( {x(1 - x)^2 - z} \right)^{m + 1} }} &= \sum_{j = 0}^m {\left( {\frac{{b_j{(z_1)} }}{{(x - z_1)^{j + 1} }} + \frac{{b_j{(z_2)} }}{{(x - z_2)^{j + 1} }} + \frac{{b_j{(z_3)} }}{{(x - z_3)^{j + 1} }}} \right)}\\ 
& = \sum_{k = 1}^3 {\sum_{j = 0}^m {\frac{{b_j{(z_k)} }}{{(x - z_k )^{j + 1} }}} } ,
\end{split}
\end{equation}
where, for $0\le r\le m$, the coefficients $b_r{(z_1)}$, $b_r{(z_2)}$, $b_r{(z_3)}$ are found from~\eqref{eq.tp25jkt}.
We therefore have
\begin{equation}\label{eq.l0f4k5d}
\begin{split}
\int_0^1 {\frac{{\ln x}}{{\left( {x(1 - x)^2 - z} \right)^{m + 1} }}\,dx} & = \sum_{k = 1}^3 {\sum_{j = 0}^m {b_j{(z_k)} \int_0^1 {\frac{{\ln x}}{{(x - z_k )^{j + 1} }}\,dx} } }\\
&=\sum_{k=1}^3{\sum_{j=0}^m{b_j{(z_k)}C_j(z_k)}},\text{ by~\eqref{eq.xhcawq8}};
\end{split}
\end{equation}
and hence~\eqref{eq.fbxhg42}.
\end{proof}

\begin{corollary}\label{cor.ldjoyyj}
Let $m$ be a non-negative integer and let $(\lambda_1,\lambda_2,\lambda_3)=(2,-i,i)$ where $i$ is the imaginary unit. Then
\begin{equation}\label{eq.pe6wz88}
\sum_{k=0}^\infty \frac{H_{3k+1} - H_{k}}{(3k+1) \binom{3k}{k}} \frac {\binom{k + m}k}{2^{k + m + 1}} =(-1)^m\sum_{k=1}^3{\sum_{j=0}^m{b_j{(\lambda_k)}C_j(\lambda_k)}}; 
\end{equation}
where, for $0\le r\le m$, the coefficients are given by
\begin{gather}
b_r{(2)} = \frac{1}{{(m - r)!}}\left. {\frac{{d^{m - r} }}{{dx^{m - r} }}\frac{1}{{(x^2 + 1)^{m + 1} }}} \right|_{x = 2} ,\label{eq.xaijx0y} \\
b_r{(-i)} = \frac{1}{{(m - r)!}}\left. {\frac{{d^{m - r} }}{{dx^{m - r}}}\frac{1}{{\left({(x - 2)(x - i)} \right)^{m + 1}}}} \right|_{x = -i},\\
b_r{(i)}  = \frac{1}{{(m - r)!}}\left. {\frac{{d^{m - r} }}{{dx^{m - r} }}\frac{1}{{\left( {(x - 2)(x + i)} \right)^{m + 1} }}} \right|_{x = i}\label{eq.zx7vo7d};
\end{gather}
and $C_r(2)$, $C_r(-i)$ and $C_r(i)$ can be readily obtained from~\eqref{eq.w9cngvk} and~\eqref{eq.y3tw4qr}.
\end{corollary}
\begin{proof}
Set $z=2$ in Theorem~\ref{thm.mtaaocu}.
\end{proof}

Here are a couple of evaluations using~\eqref{eq.pe6wz88}:
\begin{gather}
\sum_{k = 0}^\infty \frac{{H_{3k + 1} - H_k}}{(3k + 1)\binom {3k}{k}}\,\frac{{k + 1}}{{2^{k + 2} }} 
= \frac{{3\pi }}{{200}} + \frac{{\pi^2 }}{{150}} + \frac{3}{{50}}\ln 2 - \frac{4}{{125}}\ln^2 2 + \frac{{37G}}{{250}},\\
\sum_{k = 0}^\infty \frac{{H_{3k + 1} - H_k }}{(3k + 1)\binom {3k}{k}}\,\frac{{(k + 1)(k + 2)}}{{2^{k + 4} }} 
= \frac{3}{{400}} + \frac{{23\pi }}{{2500}} + \frac{{27\pi^2 }}{{10000}} + \frac{{23}}{{625}}\ln 2 - \frac{{81}}{{6250}}\ln^2 2 
+ \frac{{843G}}{{12500}},
\end{gather}
\begin{equation}
\begin{split}
&\sum_{k = 0}^\infty \frac{{H_{3k + 1} - H_k }}{(3k + 1)\binom {3k}{k}} \frac{{\binom{{k + 3}}{k}}}{{2^{k + 4} }} \\
&\qquad = \frac{{83}}{{12000}} + \frac{{3059\pi }}{{600000}} + \frac{{11\pi^2 }}{{9375}} + \frac{{1517}}{{75000}}\ln 2 
- \frac{{88}}{{15625}}\ln^2 2 + \frac{{8137G}}{{250000}}.
\end{split}
\end{equation}

\begin{corollary}\label{cor.a55j4z8}
Let $m$ be a non-negative integer and let \mbox{$(\gamma_1,\gamma_2,\gamma_3)=(-1,(3 + i\sqrt 7)/2,(3 - i\sqrt 7)/2)$} where $i$ is the imaginary unit. Then
\begin{equation}\label{eq.m3tjr6i}
\sum_{k=0}^\infty \frac{H_{3k+1} - H_{k}}{(3k+1) \binom{3k}{k}}\, \frac {(-1)^k\binom {k + m}k}{4^{k + m + 1}} =-\sum_{k=1}^3{\sum_{j=0}^m{b_j{(\gamma_k)}C_j(\gamma_k)}}; 
\end{equation}
where, for $0\le r\le m$, the coefficients are given by
\begin{gather}
b_r{(\gamma_1)}  = \frac{1}{{(m - r)!}}\left. {\frac{{d^{m - r} }}{{dx^{m - r} }}\,\frac{1}{{(x^2  - 3x + 4)^{m + 1} }}} \right|_{x = -1} ,\label{eq.pypuqrd}\\
b_r{(\gamma_2)}  = \frac{1}{{(m - r)!}}\left. {\frac{{d^{m - r} }}{{dx^{m - r} }}\,\frac{1}{{\left( {(x + 1)(x - (3 - i\sqrt 7)/2)} \right)^{m + 1} }}} \right|_{x =  (3 + i\sqrt 7)/2} ,\\
b_r{(\gamma_3)}  = \frac{1}{{(m - r)!}}\left. {\frac{{d^{m - r} }}{{dx^{m - r} }}\,\frac{1}{{\left( {(x + 1)(x - (3 + i\sqrt 7)/2)} \right)^{m + 1} }}} \right|_{x =  (3 - i\sqrt 7)/2}\label{eq.xn84iwk};
\end{gather}
\begin{equation}\label{eq.g3p0uxm}
C_0 (\gamma _1 ) = \Li_2 ( - 1) =  - \frac{{\pi ^2 }}{{12}},\quad C_0 (\gamma _2 ) = \Li_2 \left( {\frac{2}{{3 + i\sqrt 7 }}} \right),\quad C_0 (\gamma _3 ) = \Li_2 \left( {\frac{2}{{3 - i\sqrt 7 }}} \right),\text{ by~\eqref{eq.w9cngvk}};
\end{equation}
and for $1\le r\le m$, $C_r(\gamma_1)$, $C_r(\gamma_2)$ and $C_r(\gamma_3)$ are found from~\eqref{eq.y3tw4qr}.
\end{corollary}
\begin{proof}
Theorem~\ref{thm.mtaaocu} with $z=-4$.
\end{proof}

\section{Main results, Part 2}

This section deals with the second category of series, i.e., series of the form
\begin{equation}\label{eq.wl9np1c}
\sum_{k=0}^\infty \frac{H_{2k} - H_{k}}{(3k+1) \binom{3k}{k}} z^k = - \int_0^1 \frac{\ln\left (\frac{x}{1-x}\right )}{1-zx(1-x)^{2}} dx.
\end{equation}
Differentiating~\eqref{eq.wl9np1c} $m$ times with respect to $z$ and thereafter replacing $z$ with $1/z$ gives
\begin{equation*}
\sum_{k = 0}^\infty  {\frac{{H_{2k}  - H_k }}{{(3k + 1)\binom{{3k}}{k}}}\frac{{\binom{{k}}{m}}}{{z^{k + 1} }}}  =( - 1)^m \int_0^1 {\frac{{x^m (1 - x)^{2m} \ln (x/(1 - x))}}{{\left( {x(1 - x)^2  - z} \right)^{m + 1} }}}\,dx.
\end{equation*}

\begin{theorem}\label{thm.uok2s8e}
Let $m$ be a non-negative integer and let $z_1$, $z_2$ and $z_3$ be the distinct roots of $x(1 - x)^2-z=0$ where $z$ is a real number such that $|z|\ge1$. Then
\begin{equation}
\sum_{k=0}^\infty \frac{H_{2k} - H_{k}}{(3k+1) \binom{3k}{k}} \frac {\binom km}{z^{k + 1}} =(-1)^m\sum_{k=1}^3{\sum_{j=0}^m{a_j{(z_k)}\left(C_j(z_k) + (-1)^jC_j(1 - z_k)\right)}},
\end{equation}
where $a_r{(z_1)}$, $a_r{(z_2)}$, $a_r{(z_3)}$ are as defined in Theorem~\ref{thm.oxop1ki} and $C_0$ and $C_r$ can be found from~\eqref{eq.w9cngvk} and~\eqref{eq.y3tw4qr} in Theorem~\ref{thm.oxop1ki}.
\end{theorem}
\begin{proof}
The proof is similar to that of Theorem~\ref{thm.oxop1ki}. Note that
\begin{equation}
\begin{split}
\int_0^1 {\frac{{\ln (x/(1 - x))}}{{(x - \lambda )^{r + 1} }}\,dx}  &= \int_0^1 {\frac{{\ln x}}{{(x - \lambda )^{r + 1} }}\,dx}  + ( - 1)^r \int_0^1 {\frac{{\ln x}}{{(x - (1 - \lambda ))^{r + 1} }}\,dx}\\
&=C_r(\lambda) + (-1)^rC_r(1 - \lambda).
\end{split}
\end{equation}
\end{proof}
\begin{corollary}\label{cor.dw0vcc8}
Let $m$ be a non-negative integer and let $(\lambda_1,\lambda_2,\lambda_3)=(2,-i,i)$ where $i$ is the imaginary unit. Then
\begin{equation}\label{eq.q5kmtd7}
\sum_{k=0}^\infty \frac{H_{2k} - H_{k}}{(3k+1) \binom{3k}{k}} \frac {\binom km}{2^{k + 1}} =(-1)^m\sum_{k=1}^3{\sum_{j=0}^m{a_j{(\lambda_k)}\left(C_j(\lambda_k) + (-1)^jC_j(1 - \lambda_k)\right)}},
\end{equation}
where $a_r{(2)}$, $a_r{(-i)}$, $a_r{(i)}$ are as defined in \eqref{eq.zufjtr4}--\eqref{eq.c1scd26} in Corollary~\ref{cor.lofauml} and $C_r$ are found from~\eqref{eq.w9cngvk} and~\eqref{eq.y3tw4qr}.
\end{corollary}
\begin{proof}
Theorem~\ref{thm.uok2s8e} with $z=2$.
\end{proof}
Examples from Corollary~\ref{cor.dw0vcc8} include
\begin{equation}
\sum_{k = 0}^\infty  {\frac{{H_{2k}  - H_k }}{{(3k + 1)\binom{{3k}}{k}}}\,\frac{k}{{2^{k + 1} }}}  =  - \frac{\pi }{{200}} + \frac{{9\pi ^2 }}{{4000}} + \frac{3}{{100}}\ln 2 + \frac{{27}}{{1000}}\ln ^2 2 - \frac{{13}}{{1000}}\pi \ln 2,
\end{equation}
\begin{equation}
\begin{split}
&\sum_{k = 0}^\infty  {\frac{{H_{2k}  - H_k }}{{(3k + 1)\binom{{3k}}{k}}}\,\frac{{k(k - 1)}}{{2^{k + 2} }}}\\ 
&\qquad= \frac{{2\pi }}{{625}} - \frac{{149\pi ^2 }}{{100000}} - \frac{{51}}{{5000}}\ln 2 - \frac{{447}}{{25000}}\ln ^2 2 + \frac{{243}}{{25000}}\pi \ln 2.
\end{split}
\end{equation}
\begin{corollary}\label{cor.fdvgnlc}
Let $m$ be a non-negative integer and let \mbox{$(\gamma_1,\gamma_2,\gamma_3)=(-1,(3 + i\sqrt 7)/2,(3 - i\sqrt 7)/2)$} where $i$ is the imaginary unit. Then
\begin{equation}\label{eq.hwg40y4}
\sum_{k=0}^\infty \frac{H_{2k} - H_{k}}{(3k+1) \binom{3k}{k}}\, \frac {(-1)^k\binom km}{4^{k + 1}} =(-1)^{m + 1}\sum_{k=1}^3{\sum_{j=0}^m{a_j{(\gamma_k)}\left(C_j(\gamma_k) + (-1)^jC_j(1 - \gamma_k)\right)}},
\end{equation}
where $a_r{(\gamma_1)}$, $a_r{(\gamma_2)}$, $a_r{(\gamma_3)}$ are as defined in \eqref{eq.w33l0zn}--\eqref{eq.i2xtucg} in Corollary~\ref{cor.p7i3dyq} and $C_r$ are found from~\eqref{eq.w9cngvk} and~\eqref{eq.y3tw4qr}.
\end{corollary}
\begin{proof}
Theorem~\ref{thm.uok2s8e} with $z=-4$.
\end{proof}
Here we present a couple of examples from~\eqref{eq.hwg40y4}.
\begin{equation}
\begin{split}
\sum_{k = 0}^\infty  {\frac{{H_{2k}  - H_k }}{{(3k + 1)\binom{{3k}}{k}}}} \frac{{( - 1)^k k}}{{4^{k + 1} }} &= \frac{3}{{448}}\ln 2 - \frac{9}{{512}}\ln ^2 2 - \frac{3}{{128}}\arctan ^2 \left( {\frac{{\sqrt 7 }}{5}} \right)\\
&\qquad - \left( {\frac{1}{{224}} - \frac{{89}}{{6272}}\ln 2} \right)\sqrt 7 \arctan \left( {\frac{{\sqrt 7 }}{5}} \right),
\end{split}
\end{equation}
\begin{equation}
\begin{split}
\sum_{k = 0}^\infty  {\frac{{H_{2k}  - H_k }}{{(3k + 1)\binom{{3k}}{k}}}} \frac{{( - 1)^k k(k - 1)}}{{4^{k + 1} }} &=  - \frac{{219}}{{25088}}\ln 2 + \frac{{93}}{{4096}}\ln ^2 2 + \frac{{31}}{{1024}}\arctan ^2 \left( {\frac{{\sqrt 7 }}{5}} \right)\\
&\qquad + \left( {\frac{1}{{256}} - \frac{{6651}}{{{\rm 351232}}}\ln 2} \right)\sqrt 7 \arctan \left( {\frac{{\sqrt 7 }}{5}} \right).
\end{split}
\end{equation}
\bigskip

As counterpart of~\eqref{eq.kjlniph}, we have
\begin{equation}\label{eq.pwms8oq}
\sum_{k=0}^\infty \frac{H_{2k} - H_{k}}{(3k+1) \binom{3k}{k}} \frac {\binom{k + m}k}{z^{k + m + 1}} =(-1)^m\int_0^1 {\frac{{\ln \left(\frac{x}{1 - x}\right)}}{\left(x(1 - x)^2  - z\right)^{m + 1}}dx},\quad |z|\ge 1. 
\end{equation}
Note that~\eqref{eq.pwms8oq} holds for every real number $m$ that is not a negative integer.

Setting $z=2$ in~\eqref{eq.pwms8oq} gives
\begin{equation}\label{eq.lyb7lws}
\sum_{k=0}^\infty \frac{H_{2k} - H_{k}}{(3k+1) \binom{3k}{k}} \frac {\binom{k + m}k}{2^{k + m + 1}} =(-1)^m\int_0^1 {\frac{{\ln \left(\frac{x}{1 - x}\right)}}{((x - 2)(x^2  + 1))^{m + 1}}\,dx}, 
\end{equation}
while setting $z=-4$ gives
\begin{equation}\label{eq.qikg7w3}
\sum_{k=0}^\infty \frac{H_{2k} - H_{k}}{(3k+1) \binom{3k}{k}} \frac {(-1)^k\binom{k + m}k}{4^{k + m + 1}} =-\int_0^1 {\frac{{\ln \left(\frac{x}{1 - x}\right)}}{((x + 1)(x^2  -3x + 4))^{m + 1}}\,dx}. 
\end{equation}
\begin{proposition}
We have
\begin{equation}\label{eq.xuvoy6t}
\sum_{k=0}^\infty \frac{H_{2k} - H_{k}}{(3k+1) \binom{3k}{k}2^{k + 1}} = \frac{\pi \ln(2)}{20} - \frac{3}{{40}}\ln^2(2) - \frac{\pi^2}{160}.
\end{equation}
\end{proposition}
\begin{proof}
Using~\eqref{eq.f2814qz} and~\eqref{eq.yqwhh2z}, we have
\begin{equation*}
\begin{split}
\int_0^1 {\frac{{\ln \left(\frac{x}{1 - x}\right)}}{{(x - 2)(x^2  + 1)}}\,dx}  &= \frac{1}{{4i - 2}}\int_0^1 {\frac{{\ln \left(\frac{x}{1 - x}\right)}}{{x + i}}\,dx}  - \frac{1}{{4i + 2}}\int_0^1 {\frac{{\ln \left(\frac{x}{1 - x}\right)}}{{x - i}}\,dx} + \frac{1}{5}\int_0^1 {\frac{{\ln \left(\frac{x}{1 - x}\right)}}{{x - 2}}\,dx}\\
& = \frac{1}{{4i - 2}}\left( { - \frac{1}{2}\ln ^2 \left( {\frac{{ - i - 1}}{{ - i}}} \right)} \right) - \frac{1}{{4i + 2}}\left( { - \frac{1}{2}\ln ^2 \left( {\frac{{i - 1}}{i}} \right)} \right)\\
&\qquad + \frac{1}{5}\left( { - \frac{1}{2}\ln ^2 \left( {\frac{{2 - 1}}{2}} \right)} \right)\\
& = \frac{1}{{4i - 2}}\left( { - \frac{1}{2}\ln ^2 (1 - i)} \right) - \frac{1}{{4i + 2}}\left( { - \frac{1}{2}\ln ^2 (1 + i)} \right) + \frac{1}{5}\left( { - \frac{1}{2}\ln ^2 2} \right),
\end{split}
\end{equation*}
which simplifies to~\eqref{eq.xuvoy6t}.
\end{proof}

Setting $m=0$ in~\eqref{eq.qikg7w3} gives
\begin{equation}\label{eq.w2c6ek5}
\sum_{k = 0}^\infty  {\frac{{H_{2k}  - H_k }}{{(3k + 1)\binom{{3k}}{k}}}} \frac{{( - 1)^k }}{{4^{k + 1} }} =  - \int_0^1 {\frac{{\ln (x/(1 - x))dx}}{{(x + 1)(x^2  - 3x + 4)}}} .
\end{equation}
The integral occuring on the RHS can be evaluated. The result is stated in Proposition~\ref{prop.vxrxd2h}.
\begin{proposition}\label{prop.vxrxd2h}
We have
\begin{equation}\label{eq.qwvbjrj}
\begin{split}
\sum_{k = 0}^\infty  {\frac{{H_{2k}  - H_k }}{{(3k + 1)\binom{{3k}}{k}}}}\, \frac{{( - 1)^k }}{{4^{k + 1} }}& = \frac{3}{{64}}\ln ^2 2 + \frac{1}{{16}}\arctan ^2 \left( {\frac{{\sqrt 7 }}{5}} \right)\\
&\qquad- \frac{{5\sqrt 7 }}{{112}}\ln 2\,\arctan \left( {\frac{{\sqrt 7 }}{5}} \right).
\end{split}
\end{equation}
\end{proposition}
\begin{proof}
We wish to evaluate the integral in~\eqref{eq.w2c6ek5}. Proceeding as in Proposition~\ref{prop.vhpcqxx}, we have
\begin{equation*}
\begin{split}
\int_0^1 {\frac{{\ln (x/(1 - x))\,dx}}{{(x + 1)(x^2  - 3x + 4)}}}  &= \frac{1}{8}\int_0^1 {\frac{{\ln (x/(1 - x))\,dx}}{{x + 1}}}  - \frac{{7 + 5i\sqrt 7 }}{{112}}\int_0^1 {\frac{{\ln (x/(1 - x))\,dx}}{{x - (3 + i\sqrt 7 )/2}}}\\ 
&\qquad - \frac{{7 - 5i\sqrt 7 }}{{112}}\int_0^1 {\frac{{\ln (x/(1 - x))\,dx}}{{x - (3 - i\sqrt 7 )/2}}},
\end{split}
\end{equation*}
so that upon using~\eqref{eq.yqwhh2z} we obtain
\begin{equation*}
\begin{split}
\int_0^1 {\frac{{\ln (x/(1 - x))\,dx}}{{(x + 1)(x^2  - 3x + 4)}}} & =  - \frac{1}{{16}}\ln ^2 2 + \frac{{7 + 5i\sqrt 7 }}{{224}}\ln ^2 \left( {\frac{{1 + i\sqrt 7 }}{{3 + i\sqrt 7 }}} \right)\\
&\qquad + \frac{{7 - 5i\sqrt 7 }}{{224}}\ln ^2 \left( {\frac{{1 - i\sqrt 7 }}{{3 - i\sqrt 7 }}} \right)\\
&= - \frac{1}{{16}}\ln ^2 2 + \frac{1}{{112}}\Re\left( {(7 + 5i\sqrt 7 )\ln ^2 \left( {\frac{{1 + i\sqrt 7 }}{{3 + i\sqrt 7 }}} \right)} \right);
\end{split}
\end{equation*}
which simplifies to the RHS of~\eqref{eq.qwvbjrj} upon using $\Re (fg)=\Re f\Re g - \Im f\Im g$.
\end{proof}
\begin{theorem}\label{thm.kn7npyf}
Let $m$ be a non-negative integer and let $z_1$, $z_2$ and $z_3$ be the distinct roots of $x(1 - x)^2-z=0$ where $z$ is a real number such that $|z|\ge1$. Then
\begin{equation}
\sum_{k=0}^\infty \frac{H_{2k} - H_{k}}{(3k+1) \binom{3k}{k}} \frac {\binom{k + m}k}{z^{k + m + 1}} =(-1)^m\sum_{k=1}^3{\sum_{j=0}^m{b_j{(z_k)}\left(C_j(z_k) + (-1)^jC_j(1 - z_k)\right)}},
\end{equation}
where $b_r{(z_1)}$, $b_r{(z_2)}$, $b_r{(z_3)}$ are as defined in Theorem~\ref{thm.mtaaocu} and $C_0$ and $C_r$ can be found from~\eqref{eq.w9cngvk} and~\eqref{eq.y3tw4qr} in Theorem~\ref{thm.oxop1ki}.
\end{theorem}
\begin{proof}
The proof is similar to that of Theorem~\ref{thm.mtaaocu}. We evaluate the integral on the RHS of~\eqref{eq.pwms8oq}.
\end{proof}
\begin{corollary}
Let $m$ be a non-negative integer and let $(\lambda_1,\lambda_2,\lambda_3)=(2,-i,i)$ where $i$ is the imaginary unit. Then
\begin{equation}\label{eq.acolcaj}
\sum_{k=0}^\infty \frac{H_{2k} - H_{k}}{(3k+1) \binom{3k}{k}} \frac {\binom{k + m}k}{2^{k + m + 1}} =(-1)^m\sum_{k=1}^3{\sum_{j=0}^m{b_j{(\lambda_k)}\left(C_j(\lambda_k) + (-1)^jC_j(1 - \lambda_k)\right)}},
\end{equation}
where $b_r{(2)}$, $b_r{(-i)}$, $b_r{(i)}$ are as defined in \eqref{eq.xaijx0y}--\eqref{eq.zx7vo7d} in Corollary~\ref{cor.ldjoyyj} and $C_r$ are found from~\eqref{eq.w9cngvk} and~\eqref{eq.y3tw4qr}.
\end{corollary}
\begin{proof}
Theorem~\ref{thm.kn7npyf} with $z=2$.
\end{proof}
At $m=1$ and $m=2$ in~\eqref{eq.acolcaj} we obtain
\begin{equation}
\sum_{k = 0}^\infty \frac{{H_{2k} - H_k }}{{(3k + 1)\binom{{3k}}{k}}}\frac{{k + 1}}{{2^{k + 2} }} = - \frac{\pi }{{400}} + \frac{{37}}{{2000}}\pi \ln 2 - \frac{{\pi^2 }}{{500}} + \frac{3}{{200}}\ln 2 - \frac{3}{{125}}\ln^2 2
\end{equation}
and
\begin{equation}
\begin{split}
&\sum_{k = 0}^\infty \frac{{H_{2k} - H_k }}{{(3k + 1)\binom{{3k}}{k}}}\frac{{\binom{{k + 2}}{k}}}{{2^{k + 3} }} \\
&\qquad = - \frac{{17\pi }}{{10000}} + \frac{{843}}{{100000}}\pi \ln 2 - \frac{{81\pi^2 }}{{100000}} + \frac{{249}}{{20000}}\ln 2 
- \frac{{243}}{{25000}}\ln^2 2.
\end{split}
\end{equation}

\begin{corollary}
Let $m$ be a non-negative integer and let \mbox{$(\gamma_1,\gamma_2,\gamma_3)=(-1,(3 + i\sqrt 7)/2,(3 - i\sqrt 7)/2)$} where $i$ is the imaginary unit. Then
\begin{equation}\label{eq.b6t3pyn}
\sum_{k=0}^\infty \frac{H_{2k} - H_{k}}{(3k+1) \binom{3k}{k}} \frac {(-1)^k\binom{k + m}k}{4^{k + m + 1}} =
-\sum_{k=1}^3{\sum_{j=0}^m{b_j{(\gamma_k)}\left(C_j(\gamma_k) + (-1)^j C_j(1 - \gamma_k)\right)}},
\end{equation}
where $b_r{(\gamma_1)}$, $b_r{(\gamma_2)}$, $b_r{(\gamma_3)}$ are as defined in \eqref{eq.pypuqrd}--\eqref{eq.xn84iwk}
in Corollary~\ref{cor.a55j4z8} and $C_r$ are found from~\eqref{eq.w9cngvk} and~\eqref{eq.y3tw4qr}.
\end{corollary}
\begin{proof}
Use Theorem~\ref{thm.kn7npyf} with $z=-4$.
\end{proof}

\section{Concluding remarks}

Theorems similar to those in the previous section can be stated for alternating sums. There does not appear to be values of $z$, however,  
for which the integrals can be evaluated in terms of elementary functions. For a slightly different direction of future research we mention that
replacing $z$ by $iz$ and comparing the real and imaginary parts we get integral relations of the form
\begin{equation}
\sum_{k=0}^\infty (-1)^k \frac{H_{6k+1} - H_{2k}}{(6k+1) \binom{6k}{2k}} z^k = - \int_0^1 \frac{\ln(x)}{1+z^2 x^2(1-x)^4} dx,
\end{equation}
\begin{equation}
\sum_{k=0}^\infty (-1)^k \frac{H_{6k+4} - H_{2k+1}}{(6k+4) \binom{6k+3}{2k+1}} z^k = - z \int_0^1 x(1-x)^2 \frac{\ln(x)}{1+z^2 x^2(1-x)^4} dx,
\end{equation}
\begin{equation}
\sum_{k=0}^\infty (-1)^k \frac{H_{4k} - H_{2k}}{(6k+1) \binom{6k}{2k}} z^k = 
- \int_0^1 \frac{\ln\left (\frac{x}{1-x}\right )}{1+z^2 x^2(1-x)^4} dx,
\end{equation}
and
\begin{equation}
\sum_{k=0}^\infty (-1)^k \frac{H_{4k+2} - H_{2k+1}}{(6k+4) \binom{6k+3}{2k+1}} z^k = 
- z \int_0^1 x(1-x)^2 \frac{\ln\left (\frac{x}{1-x}\right )}{1+z^2 x^2(1-x)^4} dx.
\end{equation}

In all cases one can attempt to evaluate the integrals appearing on the right hand sides.

\end{document}